\documentclass[12pt]{amsart}
\usepackage{color}
\usepackage{enumerate}
\usepackage{graphics}
\usepackage[colorlinks, urlcolor=blue]{hyperref}
\usepackage{mathrsfs}
\usepackage{amssymb}
\usepackage[all]{xy}
\usepackage{amsthm}

\newtheorem{theorem}{Theorem}[section]
\newtheorem*{theorem*}{Theorem}
\newtheorem*{problem*}{Problem}
\newtheorem{lemma}[theorem]{Lemma}
\newtheorem{proposition}[theorem]{Proposition}
\newtheorem{corollary}[theorem]{Corollary}
\newtheorem{definition}[theorem]{Definition}

\theoremstyle{remark}

\newtheorem{remark}[theorem]{Remark}
\newtheorem{Problem}[theorem]{Problem}

\begin{document}

\title[Schur multipliers]{Schur multipliers on $\mathcal{B}(L^p,L^q)$}
 
\author[C. Coine]{Cl\'ement Coine}
\email{clement.coine@univ-fcomte.fr}

\address{Laboratoire de Mathématiques de Besan\c con, UMR 6623, CNRS, Universit\'e Bourgogne Franche-Comt\'e,
25030 Besan\c{c}on Cedex, FRANCE}

\maketitle

\begin{abstract}
Let $(\Omega_1, \mathcal{F}_1, \mu_1)$ and $(\Omega_2, \mathcal{F}_2, \mu_2)$ be two measure spaces and let $1 \leq p,q \leq +\infty$. We give a definition of Schur multipliers on $\mathcal{B}(L^p(\Omega_1), L^q(\Omega_2))$ which extends the definition of classical Schur multipliers on $\mathcal{B}(\ell_p,\ell_q)$. Our main result is a characterization of Schur multipliers in the case $1\leq q \leq p \leq +\infty$. When $1 < q \leq p < +\infty$, $\phi \in L^{\infty}(\Omega_1 \times \Omega_2)$ is a Schur multiplier on $\mathcal{B}(L^p(\Omega_1), L^q(\Omega_2))$ if and only if there are a measure space (a probability space when $p\neq q$) $(\Omega,\mu)$, $a\in L^{\infty}(\mu_1, L^{p}(\mu))$ and $b\in L^{\infty}(\mu_2, L^{q'}(\mu))$ such that, for almost every $(s,t) \in \Omega_1  \times \Omega_2$,
$$\phi(s,t)=\left\langle a(s), b(t) \right\rangle.$$
Here, $L^{\infty}(\mu_1, L^{r}(\mu))$ denotes the Bochner space on $\Omega_1$ valued in $L^r(\mu)$. This result is new, even in the classical case. As a consequence, we give new inclusion relationships between the spaces of Schur multipliers on $\mathcal{B}(\ell_p,\ell_q)$.
\end{abstract}

\section{Introduction}
If $1 \leq r < +\infty$, we denote by $\ell_r$ the Banach space of $r-$summable sequences $(x_i)_{i\geq 1} \subset \mathbb{C}$ (that is, $\sum_i |x_i|^r < +\infty$) endowed with the norm $\|x\|_{\ell_r}=\left( \sum_i |x_i|^r \right)^{1/r}$. Let $\ell_{\infty}$ be the Banach space of bounded sequences $(y_i)_{i\geq 1} \subset \mathbb{C}$ with the norm $\|y\|_{\ell_{\infty}}=\sup_i |y_i|$. If $n\in \mathbb{N}$, we denote by $\ell_r^n$ the $n-$dimensional versions of the spaces introduced before.\\

Let $m=(m_{ij})_{i,j \geq 1}$ be a bounded family of complex numbers and let $1 \leq p,q \leq +\infty$. We say that $m$ is a Schur multiplier on $\mathcal{B}(\ell_p,\ell_q)$ if for any matrix $[a_{ij}]_{i,j \geq 1}$ in $\mathcal{B}(\ell_p,\ell_q)$, the matrix $[m_{ij}a_{ij}]_{i,j \geq 1}$ defines an element of $\mathcal{B}(\ell_p,\ell_q)$. An application of the Closed Graph theorem shows that $m$ is a Schur multiplier if and only if the mapping
\begin{equation}\label{MapSchur}
\begin{array}[t]{lrcl}
& T_m : \mathcal{B}(\ell_p,\ell_q) & \longrightarrow & \mathcal{B}(\ell_p,\ell_q) \\
&  [a_{ij}]_{i,j \geq 1} & \longmapsto & [m_{ij}a_{ij}]_{i,j \geq 1} \end{array}
\end{equation}
is bounded. By definition, the norm of the Schur multiplier $m$ is the norm of $T_m$.\\

There is a well-known characterization of Schur multipliers on $\mathcal{B}(\ell_2)$ (see for instance \cite[Theorem 5.1]{PisierBook2}) which can be extended to the case $\mathcal{B}(\ell_p)$ as follows.

\begin{theorem}\label{FactoPisier}\cite[Theorem 5.10]{PisierBook2}
Let $\phi = (c_{ij})_{i,j\in \mathbb{N}} \subset \mathbb{C}$, $C \geq 0$ be a constant and let $1\leq p < \infty$. The following are equivalent :
\begin{enumerate}
\item[(i)] $\phi$ is a Schur multiplier on $\mathcal{B}(\ell_p)$ with norm $\leq C$.
\item[(ii)] There is a measure space $(\Omega, \mu)$ and elements $(x_j)_{j\in \mathbb{N}}$ in $L^p(\mu)$ and $(y_i)_{i\in \mathbb{N}}$ in $L^{p'}(\mu)$ such that
\begin{equation*}
\forall i,j \in \mathbb{N}, \ c_{ij}=\left\langle x_j,y_i \right\rangle \ \text{and} \ \underset{i}{\sup} \|y_i\|_{p'} \ \underset{j}{\sup} \|x_j\|_p \leq C.
\end{equation*}
\end{enumerate}
\end{theorem}

Denote by $\mathcal{M}(p,q)$ the space of Schur multipliers on $\mathcal{B}(\ell_p,\ell_q)$. In \cite{Bennett}, Bennett gives some results about the inclusions between the spaces $\mathcal{M}(p,q)$. In the same paper, he also gives a necessary and sufficient condition for a family $m$ to belong to $\mathcal{M}(p,q)$, using the theory of absolutely summing operators. Theorem $\ref{FactoPisier}$ provides a different type of characterization, which is more explicit and useful.\\

Let $(\Omega_1, \mu_1)$ and $(\Omega_2, \mu_2)$ be two $\sigma$-finite measure spaces. The space $L^2(\Omega_1 \times \Omega_2)$ can be identified with the space $S^2(L^2(\Omega_1), L^2(\Omega_2))$ of Hilbert-Schmidt operators. If $J\in L^2(\Omega_1 \times \Omega_2)$, the operator
\begin{equation*}
\begin{array}[t]{lccc}
X_J : & L^2(\Omega_1) & \longrightarrow & L^2(\Omega_2) \\
& f & \longmapsto & \displaystyle \int_{\Omega_1} J(t,\cdot)f(t)\, \text{d}\mu_1(t)  \end{array}
\end{equation*}
is a Hilbert-Schmidt operator and $\|X_J\|_2=\|J\|_{L^2}$. Moreover, 
any element of $S^2(L^2(\Omega_1), L^2(\Omega_2))$ has this form.

Let $\phi\in L^{\infty}(\Omega_1\times \Omega_2)$. We may associate the operator
\begin{equation*}
\begin{array}[t]{lccc}
R_{\phi} : & S^2(L^2(\Omega_1), L^2(\Omega_2)) & \longrightarrow & S^2(L^2(\Omega_1), L^2(\Omega_2)) \\
& X_J & \longmapsto & X_{\phi J}  \end{array}
\end{equation*}
whose norm is equal to $\|\phi\|_{\infty}$.
We say that $\phi$ is a Schur multiplier on $\mathcal{B}(L^2(\Omega_1), L^2(\Omega_2))$ if $R_{\psi}$ extends to a (necessarily unique) bounded operator still denoted by
$$
R_{\phi} \colon \mathcal{K}(L^2(\Omega_1), L^2(\Omega_2)) 
\longrightarrow \mathcal{K}(L^2(\Omega_1), L^2(\Omega_2)),
$$
where $\mathcal{K}(L^2(\Omega_1), L^2(\Omega_2))$ denotes the space of compact operators from
$L^2(\Omega_1)$ into $L^2(\Omega_2)$. When $\phi$ is a Schur multiplier, the norm of $\phi$ is by definition the norm of $R_{\phi}$ as an operator from $\mathcal{K}(L^2(\Omega_1), L^2(\Omega_2))$ into itself.

A characterization similar to the one in Theorem $\ref{FactoPisier}$ holds in this setting. The following result was established by Peller \cite{PelHank}.

\begin{theorem}
Let $\phi \in L^{\infty}(\Omega_1 \times \Omega_2)$ and $C > 0$. The following are equivalent :
\begin{enumerate}
\item[(i)] $\phi$ is a Schur multiplier and $\|R_{\phi}\| < C$.
\item[(ii)] There exist families $(a_i)_{i \geq 1} \subset L^{\infty}(\Omega_1)$ and $(b_i)_{i \geq 1} \subset L^{\infty}(\Omega_2)$ such that
$$\underset{s \in \Omega_1}{\text{essup}} \sum_{i=1}^{+\infty} |a_i(s)|^2 < C, \underset{t \in \Omega_2}{\text{essup}} \sum_{i=1}^{+\infty} |b_i(t)|^2 < C$$
and for almost every $(s,t)\in \Omega_1 \times \Omega_2$,
$$\phi(s,t) = \sum_{i=1}^{+\infty} a_i(s)b_i(t).$$
\end{enumerate}
\end{theorem}

\noindent See also \cite{Spronk} for another formulation of this theorem and results about Schur multipliers in the measurable case.\\

In this article, we define more generally Schur multipliers on $\mathcal{B}(L^p(\Omega_1), L^q(\Omega_2))$ for some measure spaces $(\Omega_1,\mu_1)$ and $(\Omega_2, \mu_2)$. To any $\phi \in L^{\infty}(\Omega_1, \Omega_2)$, we associate a linear mapping
$$T_{\phi} : L^{p'}(\Omega_1) \otimes L^q(\Omega_2) \rightarrow L^{p'}(\Omega_1) \overset{\vee}{\otimes} L^q(\Omega_2)$$
and we say that $\phi$ is a Schur multiplier if $T_{\phi}$ is bounded. When $\Omega_1 = \Omega_2=\mathbb{N}$ with the counting measures, $T_{\phi}$ corresponds to $(\ref{MapSchur})$.\\
In the case $1 \leq q \leq p \leq +\infty$, we characterize the elements of $L^{\infty}(\Omega_1 \times \Omega_2)$ which are Schur multipliers on $\mathcal{B}(L^p(\Omega_1), L^q(\Omega_2))$. We prove that if $1 < q \leq p < +\infty$, $\phi$ is a Schur multiplier if and only if there are a measure space (a probability space when $p\neq q$) $(\Omega,\mu)$, $a\in L^{\infty}(\mu_1, L^{p}(\mu))$ and $b\in L^{\infty}(\mu_2, L^{q'}(\mu))$ such that, for almost every $(s,t) \in \Omega_1  \times \Omega_2$,
$$\phi(s,t)=\left\langle a(s), b(t) \right\rangle,$$
where $L^{\infty}(\mu_1, L^{r}(\mu))$ is the Bochner space valued in $L^r(\mu)$.\\
This result is new, even in the setting of classical Schur multipliers on $\mathcal{B}(\ell_p,\ell_q)$, and is of different nature than the characterization of Bennett. As a consequence, we give in the last section of this article new results of comparisons for the spaces $\mathcal{M}(p,q)$.

\subsection{Notations}\label{Notations}

Let $X$ and $Y$ be Banach spaces. \\

If $z\in X \otimes Y$, the projective tensor norm of $z$ is defined by
$$
\|z\|_{\wedge} := \inf \left\lbrace \sum \|x_i\| \|y_i\| \right\rbrace,
$$
where the infimum runs over all finite families $(x_i)_i$ in $X$ and $(y_i)_i$ in $Y$ such 
that 
$$
z=\sum_i x_i\otimes y_i. 
$$
The completion $X \overset{\wedge}{\otimes} Y$ of $(X\otimes Y,\|.\|_{\wedge})$
is called the projective tensor product of $X$ and $Y$. Note that the projective tensor product is commutative, that is $X \overset{\wedge}{\otimes} Y = Y \overset{\wedge}{\otimes} X$.

The mapping taking any functional $\omega\colon X\otimes Y\to\mathbb{C}$
to the operator $u\colon X\to Y^*$ defined by $\langle u(x),
y\rangle=\omega(x\otimes y)$ for any $x\in X, y\in Y$,
induces an isometric identification
\begin{equation}\label{dualproj}
(X \overset{\wedge}{\otimes} Y)^*=\mathcal{B}(X,Y^*).
\end{equation}
We refer to \cite[Chapter 8, Corollary 2]{DU} for this fact.

Let $(\Omega,\mu)$ be a localizable measure space and let $L^p(\Omega;Y)$ denote the
Bochner space of $p-$integrable functions from $\Omega$ into $Y$.
By \cite[Chapter 8, Example 10]{DU}, the natural embedding $L^1(\Omega)
\otimes Y\subset L^1(\Omega;Y)$
extends to an isometric isomorphism
\begin{equation}\label{L1tensor}
L^1(\Omega;Y) = L^1(\Omega)\overset{\wedge}{\otimes}Y.
\end{equation}
By (\ref{dualproj}), this implies
\begin{equation}\label{L1tensorcor}
L^1(\Omega;Y)^*=\mathcal{B}(L^1(\Omega),Y^*).
\end{equation}
Assume that $Y^*$ has the Radon-Nikodym property (in short, $Y^*$ has RNP). In this case,
$$L^1(\Omega, Y)^* = L^{\infty}(\Omega, Y^*).$$
The latter implies that 
\begin{equation}\label{Linftensorcor}
L^{\infty}(\Omega, Y^*) = \mathcal{B}(L^1(\Omega),Y^*),
\end{equation}
and the isometric isomorphism is given by
\begin{equation*}
\begin{array}[t]{lrcl}
& L^{\infty}(\Omega, Y^*) & \longrightarrow & \mathcal{B}(L^1(\Omega),Y^*). \\
&   g & \longmapsto & \left[f \in L^1(\Omega) \mapsto \displaystyle{\int_{\Omega} f(t)g(t) \text{d}\mu(t)} \right] \end{array}
\end{equation*}
Assume now that $Y=L^1(\Omega')$ where $(\Omega', \mu')$ is a localizable measure space. Then, an application of Fubini Theorem gives
$$L^1(\Omega, L^1(\Omega'))=L^1(\Omega \times \Omega').$$
Using equality $(\ref{L1tensor})$, we deduce that
\begin{equation}\label{LinfB}
\mathcal{B}(L^1(\Omega), L^{\infty}(\Omega')) = L^{\infty}(\Omega \times \Omega'), 
\end{equation}
and the correspondence is given by
\begin{equation*}
\begin{array}[t]{lrcl}
& L^{\infty}(\Omega \times \Omega') & \longrightarrow & \mathcal{B}(L^1(\Omega), L^{\infty}(\Omega')). \\
&   \psi & \longmapsto & \left[f \in L^1(\Omega) \mapsto \displaystyle{\int_{\Omega} f(t)\psi(t, \cdot) \text{d}\mu(t)} \right] \end{array}
\end{equation*}
For $\psi \in L^{\infty}(\Omega \times \Omega')$, denote by $u_{\psi}$ the corresponding element of $\mathcal{B}(L^1(\Omega), L^{\infty}(\Omega'))$.\\

If $z = \sum_i x_i \otimes y_i \in X \otimes Y$, $x^* \in X^*$ and $y^* \in Y^*$, we write
$$\left\langle z , x^* \otimes y^* \right\rangle = \sum_i x^*(x_i)y^*(y_i).$$
Then, the injective tensor norm of $z \in X \otimes Y$ is given by
$$\|z\|_{\vee} = \underset{\|x^*\| \leq 1, \|y^*\| \leq 1}{\sup} |\left\langle z , x^* \otimes y^* \right\rangle|.$$
The completion $X \overset{\vee}{\otimes} Y$ of $(X\otimes Y,\|.\|_{\vee})$ is called the injective tensor product of $X$ and $Y$.\\

In this paper, we will often identify $X^*\otimes Y$ with the finite rank operators from $X$ into $Y$ as follow. If $u=\sum_i x_i^* \otimes y_i\in X^*\otimes Y$, we define $\tilde{u} : X \rightarrow Y$ by
\begin{equation}\label{tensorop}
\tilde{u}(x) = \sum_i x_i^*(x) y_i, \forall x \in X.
\end{equation}
Then, it is easy to check that $\|u\|_{\vee} = \| \tilde{u}\|_{\mathcal{B}(X, Y)}$.

Moreover, if $Y$ has the approximation property (see e.g. \cite{DiesFourBook} for the definition), \cite[Theorem 1.4.21]{DiesFourBook} gives the isometric identification
$$X^* \overset{\vee}{\otimes} Y = \mathcal{K}(X,Y)$$
where $\mathcal{K}(X,Y)$ denotes the space of compact operators from $X$ into $Y$.

Let $(\Omega_1, \mathcal{F}_1, \mu_1)$ and $(\Omega_2, \mathcal{F}_2, \mu_2)$ be two localizable measure spaces. Let $1\leq p < \infty$ and $1 \leq q \leq \infty$. Then $L^q(\Omega_2)$ has the approximation property so that we have
\begin{equation}\label{injcompact}
L^{p'}(\Omega_1) \overset{\vee}{\otimes} L^q(\Omega_2)=\mathcal{K}(L^p(\Omega_1),L^q(\Omega_2)).
\end{equation}

\noindent Finally, if we assume that $1 < p,q < +\infty$, then by \cite[Theorem 2.5]{DiesFourArt} and $(\ref{dualproj})$, 
\begin{equation}\label{bidualcomp}
(L^{p'}(\Omega_1) \overset{\vee}{\otimes} L^q(\Omega_2))^{**}
 = (L^p(\Omega_1) \overset{\wedge}{\otimes} L^{q'}(\Omega_2))^* 
 =\mathcal{B}(L^p(\Omega_1),L^q(\Omega_2)).
\end{equation}

\vspace{0.3cm}

\section{Definition of Schur multipliers on $\mathcal{B}(L^p,L^q)$}\label{first}

Let $(\Omega_1, \mathcal{F}_1, \mu_1)$ and $(\Omega_2, \mathcal{F}_2, \mu_2)$ be two localizable measure spaces and let $\phi\in L^{\infty}(\Omega_1 \times \Omega_2)$. Let $1 \leq p,q \leq \infty$ and denote by $p'$ and $q'$ their conjugate exponents.\\
Let
$$T_{\phi} : L^{p'}(\Omega_1) \otimes L^q(\Omega_2) \rightarrow \mathcal{B}(L^p(\Omega_1),L^q(\Omega_2))$$
be defined for any elementary tensor $f\otimes g\in L^{p'}(\Omega_1) \otimes L^q(\Omega_2)$ by
$$[T_{\phi}(f \otimes g)](h)= \left( \int_{\Omega_1} \phi(s,\cdot)f(s)h(s) \text{d}\mu_1(s) \right) g(\cdot) \in L^q(\Omega_2),$$
for all $h\in L^p(\Omega_1)$.\\

\noindent We have an inclusion
$$L^{p'}(\Omega_1) \otimes L^q(\Omega_2) \subset L^{p'}(\Omega_1, L^q(\Omega_2))$$
given by $f \otimes g \mapsto [s \in \Omega_1 \mapsto f(s)g]$. Under this identification, $T_{\phi}$ is the multiplication by $\phi$. Note that $L^{p'}(\Omega_1, L^q(\Omega_2))$ is invariant by multiplication by an element of $L^{\infty}(\Omega_1 \times \Omega_2)$ and that we have a contractive inclusion
$$L^{p'}(\Omega_1, L^q(\Omega_2)) \subset L^{p'}(\Omega_1) \overset{\vee}{\otimes} L^q(\Omega_2).$$
Therefore, $T_{\phi}$ is valued is in $L^{p'}(\Omega_1) \overset{\vee}{\otimes} L^q(\Omega_2)$. Using the identification
$$L^{p'}(\Omega_1) \overset{\vee}{\otimes} L^q(\Omega_2) \subset \mathcal{B}(L^p(\Omega_1), L^q(\Omega_2))$$
given by $(\ref{tensorop})$, we deduce that the elements of $L^{p'}(\Omega_1) \overset{\vee}{\otimes} L^q(\Omega_2)$ are compact operators as limits of finite rank operators for the operator norm.

\begin{definition}
We say that $\phi$ is a Schur multiplier on $\mathcal{B}(L^p(\Omega_1), L^q(\Omega_2))$ if there exists a constant $C \geq 0$ such that for all $u\in L^{p'}(\Omega_1) \otimes L^q(\Omega_2)$,
$$\|T_{\phi}(u)\|_{\mathcal{B}(L^p(\Omega_1), L^q(\Omega_2))} \leq \|u\|_{\vee},$$
that is, if $T_{\phi}$ extends to a bounded operator
$$T_{\phi} : L^{p'}(\Omega_1) \overset{\vee}{\otimes} L^q(\Omega_2) \rightarrow L^{p'}(\Omega_1) \overset{\vee}{\otimes} L^q(\Omega_2).$$
In this case, the norm of $\phi$ is by definition the norm of $T_{\phi}$.
\end{definition}

\begin{remark}\label{simplef}
By $\mathcal{E}_1$ (resp. $\mathcal{E}_2$) we denote the space of simple functions on $\Omega_1$ (resp. $\Omega_2$). By density of $\mathcal{E}_1 \otimes \mathcal{E}_2$ in $L^{p'}(\Omega_1) \overset{\vee}{\otimes} L^q(\Omega_2)$, $T_{\phi}$ extends to a bounded operator from $L^{p'}(\Omega_1) \overset{\vee}{\otimes} L^q(\Omega_2)$ into itself if and only if it is bounded on $\mathcal{E}_1 \otimes \mathcal{E}_2$ equipped with the injective tensor norm.
\end{remark}

\noindent Assume that $1 < p,q < +\infty$. By $(\ref{injcompact})$ we have
$$L^{p'}(\Omega_1) \overset{\vee}{\otimes} L^q(\Omega_2)=\mathcal{K}(L^p(\Omega_1),L^q(\Omega_2)),$$
so that $\phi$ is a Schur multiplier on $\mathcal{B}(L^p(\Omega_1), L^q(\Omega_2))$ if and only if $T_{\phi}$ extends to a bounded operator
$$T_{\phi} : \mathcal{K}(L^p(\Omega_1),L^q(\Omega_2)) \rightarrow \mathcal{K}(L^p(\Omega_1),L^q(\Omega_2)).$$
In this case, considering the bi-adjoint of $T_{\phi}$, we obtain by $(\ref{bidualcomp})$ a $w^*-$continuous mapping
$$\tilde{T_{\phi}} : \mathcal{B}(L^p(\Omega_1),,L^q(\Omega_2)) \rightarrow \mathcal{B}(L^p(\Omega_1),L^q(\Omega_2))$$
which extends $T_{\phi}$. This explains the terminology '$\phi$ is a Schur multiplier on $\mathcal{B}(L^p(\Omega_1),L^q(\Omega_2))$'.\\

\noindent \textbf{Classical Schur multipliers :} Assume that $\Omega_1=\Omega_2=\mathbb{N}$ and that $\mu_1$ and $\mu_2$ are the counting measures. An element $\phi \in L^{\infty}(\mathbb{N}^2)$ is given by a family $c=(c_{ij})_{i,j\in \mathbb{N}}$ of complex numbers, where $c_{ij}=\phi(j,i)$. In this situation, the mapping $T_{\phi}$ is nothing but the classical Schur multiplier
$$A=[a_{ij}]_{i,j \geq 1} \in \mathcal{B}(\ell_p,\ell_q) \longmapsto [c_{ij}a_{ij}]_{i,j \geq 1} .$$
When this mapping is bounded from $\mathcal{B}(\ell_p,\ell_q)$ into itself, we will denote it by $T_c$.\\

\noindent \textbf{Notations :} If $(\Omega, \mathcal{F}, \mu)$ is a measure space and $n\in \mathbb{N}^*$, we denote by $\mathcal{A}_{n, \Omega}$ the collection of $n-$tuples $(A_1, \ldots, A_n)$ of pairwise disjoint elements of $\mathcal{F}$ such that
$$\text{for all} \ 1\leq i \leq n, 0 < \mu(A_i) < + \infty.$$
If $A=(A_1, \ldots, A_n) \in \mathcal{A}_{n, \Omega}$ and $1\leq p\leq +\infty$, denote by $S_{A, p}$ the subspace of $L^p(\Omega)$ generated by $\chi_{A_1}, \ldots, \chi_{A_n}$. Then $S_{A, p}$ is $1-$complemented in $L^p(\Omega)$, and a norm one projection from $L^p(\Omega)$ into $S_{A, p}$ is given by the conditional expectation
\begin{equation}\label{conexp}
\begin{array}[t]{lrcl}
P_{A, p} : & L^p(\Omega) & \longrightarrow & L^p(\Omega). \\
& f & \longmapsto & \displaystyle{\sum_{i=1}^n} \ \dfrac{1}{\mu(A_i)} \left( \int_{A_i}f \right) \chi_{A_i} \end{array}
\end{equation}
Note that the mapping
\begin{equation}\label{isolp}
\begin{array}[t]{lrcl}
\varphi_{A, p} : & S_{A, p} & \longrightarrow & \ell_p^n. \\
&  f=\sum_i a_i \chi_{A_i} & \longmapsto & (a_i (\mu_1(A_i))^{1/p})_{i=1}^n \end{array}
\end{equation}
is an isometric isomorphism between $S_{A, p}$ and $\ell_p^n$.

\begin{proposition}\label{discret}
Let $(\Omega_1, \mathcal{F}_1, \mu_1)$ and $(\Omega_2, \mathcal{F}_2, \mu_2)$ be two measure spaces and let $\phi\in L^{\infty}(\Omega_1 \times \Omega_2)$. The following are equivalent :
\begin{enumerate}
\item[(i)] $\phi$ is a Schur multiplier on $\mathcal{B}(L^p(\Omega_1), L^q(\Omega_2))$.
\item[(ii)] For all $n,m \in \mathbb{N}^*$, for all $A=(A_1,\ldots,A_n) \in \mathcal{A}_{n,\Omega_1}, B=(B_1,\ldots,B_m) \in \mathcal{A}_{m,\Omega_2}$, write
$$\phi_{ij} = \dfrac{1}{\mu_1(A_j)\mu_2(B_i)} \displaystyle \int_{A_j \times B_i} \phi \ \text{d}\mu_1 \text{d}\mu_2.$$
Then the Schur multipliers on $\mathcal{B}(\ell_p^n, \ell_q^m)$ associated with the families $\phi_{A,B}=(\phi_{ij})$ are uniformly bounded with respect to $n,m, A$ and $B$. 
\end{enumerate}
In this case, $\|T_{\phi}\| = sup_{n,m, A, B} \|T_{\phi_{A,B}}\| < +\infty$.
\end{proposition}

\begin{proof}
$(i) \Rightarrow (ii)$. Assume first that $\phi$ is a Schur multiplier on $\mathcal{B}(L^p(\Omega_1), L^q(\Omega_2))$ with $\|T_{\phi}\| \leq 1$. Let $n,m \in \mathbb{N}^*, A=(A_1,\ldots,A_n) \in \mathcal{A}_{n,\Omega_1}$ and $B=(B_1,\ldots,B_m) \in \mathcal{A}_{m,\Omega_2}$. Let $c = \sum_{i,j} c(i,j) e_j \otimes e_i \in \ell_{p'}^n \otimes \ell_q^m \simeq \mathcal{B}(\ell_p^n, \ell_q^m)$.\\
Let $\varphi_{A,p} : S_{A, p} \rightarrow \ell_p^n$ and $\psi_{B,q} : S_{B, q} \rightarrow \ell_q^m$ be the isometries defined in $(\ref{isolp})$. Then $\tilde{c} := \psi_{B,q}^{-1} \circ c \circ \varphi_{A,p} : S_{A, p} \rightarrow S_{B, q}$ satisfies $\|\tilde{c}\| = \|c\|$ and we have
\begin{align*}
\tilde{c}
& = \sum_{i,j} \dfrac{c(i,j)}{\mu_1(A_j)^{1/p'}\mu_2(B_i)^{1/q}} \chi_{A_j} \otimes \chi_{B_i}\\
& := \sum_{i,j} \tilde{c}(i,j) \chi_{A_j} \otimes \chi_{B_i},
\end{align*}
where $\tilde{c}(i,j) = \dfrac{c(i,j)}{\mu_1(A_j)^{1/p'}\mu_2(B_i)^{1/q}}$.\\
The operator $u := \psi_{B,q} \circ P_{B, q} \circ T_{\phi}(\tilde{c})_{| S_{A, p}} \circ \varphi_{A,p}^{-1} : \ell_p^n \rightarrow \ell_q^m$ satisfies
$$ \| u\| \leq \| T_{\phi}(\tilde{c}) \|$$
and by assumption
$$\|T_{\phi}(\tilde{c})\| \leq \|\tilde{c}\|$$
so that
\begin{equation}\label{equa1}
\|u\| \leq \|\tilde{c}\| = \|c\|.
\end{equation}
Let us prove that $u = T_{\phi_{A,B}}(c)$ where $T_{\phi_{A,B}}$ is the Schur multiplier associated with the family $(\phi_{ij})$.\\
Write $u(i,j):=\psi_{B,q} \circ P_{B, q} \circ T_{\phi}(\chi_{A_j} \otimes \chi_{B_i})_{| S_{A, p}} \circ \varphi_{A,p}^{-1}$. We have
$$ u = \sum_{i,j} \tilde{c}(i,j) u(i,j).$$
Let $1 \leq k \leq n$.
\begin{align*}
[u(i,j)](e_k)
& = [\psi_{B,q} \circ P_{B, q} \circ T_{\phi}(\chi_{A_j} \otimes \chi_{B_i})_{| S_{A,p}}]\left(\dfrac{1}{\mu_1(A_k)^{1/p}} \chi_{A_k}\right)\\
& = \dfrac{1}{\mu_1(A_k)^{1/p}}  [\psi_{B,q} \circ P_{B, q}]\left( \chi_{B_i}(\cdot) \int_{\Omega_1} \phi(s,\cdot) \chi_{A_j}(s)\chi_{A_k}(s) \text{d}\mu_1(s)\right)
\end{align*}
so that $[u(i,j)](e_k)=0$ if $k\neq j$ and if $k = j$ then
\begin{align*}
[u(i,j)](e_k)
& = \dfrac{1}{\mu_1(A_k)^{1/p}} [\psi_{B,q} \circ P_{B, q}]\left(\chi_{B_i}(\cdot) \int_{A_j} \phi(s,\cdot) \text{d}\mu_1(s) \right)\\
& = \dfrac{1}{\mu_1(A_k)^{1/p}\mu_2(B_i)} \left( \int_{A_j \times B_i} \phi \right) \psi_q(\chi_{B_i})\\
& = \dfrac{1}{\mu_1(A_k)^{1/p}\mu_2(B_i)^{1/q'}} \left( \int_{A_j \times B_i} \phi \right) e_i
\end{align*}
It follows that
\begin{align*}
u
& = \sum_{i,j} \dfrac{c(i,j)}{\mu_1(A_j)^{1/p'}\mu_2(B_i)^{1/q}}\dfrac{1}{\mu_1(A_j)^{1/p}\mu_2(B_i)^{1/q'}} \left( \int_{A_j \times B_i} \phi \right) e_j \otimes e_i\\
& = \sum_{i,j} \dfrac{c(i,j)}{\mu_1(A_j)\mu_2(B_i)} \left( \int_{A_j \times B_i} \phi \right) e_j \otimes e_i\\
& = \sum_{i,j} \phi_{ij} c(i,j) e_j \otimes e_i
\end{align*}
that is, $u=T_{\phi_{A,B}}(c)$. We conclude thanks to the inequality $(\ref{equa1})$.\\

$(ii) \Rightarrow (i)$. Assume now that the assertion $(ii)$ is satisfied and show that $\phi$ is a Schur multiplier. By Remark $\ref{simplef}$, we just need to show that $T_{\phi}$ is bounded on $\mathcal{E}_1 \otimes \mathcal{E}_2$. Let $v \in \mathcal{E}_1 \otimes \mathcal{E}_2$ and write $\alpha=sup_{n,m, A, B} \|T_c\|$. We will show that $\|T_{\phi}(v)\| \leq \alpha \|v\|$. By density, it is enough to prove that for any $h_1\in \mathcal{E}_1, h_2 \in \mathcal{E}_2$,
\begin{equation}\label{equa2}
|\left\langle [T_{\phi}(v)](h_1), h_2 \right\rangle_{L^q, L^{q'}} | \leq \alpha \|v\|_{\mathcal{B}(L^p(\Omega_1), L^q(\Omega_2))} \|h_1\|_{L^p(\Omega_1)}\|h_2\|_{L^{q'}(\Omega_2)}.
\end{equation}
By assumption, there exist $n,m\in \mathbb{N}^*, A=(A_1,\ldots,A_n) \in \mathcal{A}_{n,\Omega_1}, B=(B_1,\ldots,B_m) \in \mathcal{A}_{m,\Omega_2}$ and complex numbers $v(i,j), a_i, b_j$ such that
$$v = \sum_{i,j} v(i,j) \chi_{A_j}\otimes \chi_{B_i}, h_1 = \sum_j a_j \chi_{A_j} \ \text{and} \ h_2=\sum_i b_i \chi_{B_i}.$$
Equation $(\ref{equa2})$ can be rewritten as
\begin{equation}\label{equa3}
\left| \sum_{i,j} v(i,j)a_jb_i \left( \int_{A_j \times B_i} \phi \right) \right| \leq \alpha \|v\| \|h_1\|_{L^p(\Omega_1)}\|h_2\|_{L^{q'}(\Omega_2)}.
\end{equation}
Consider $\tilde{v}:= \psi_{B,q} \circ v \circ \varphi_{A,p}^{-1} : \ell_p^n \rightarrow \ell_q^m$ and $z := \psi_{B,q} \circ P_{B, q} \circ T_{\phi}(v)_{| S_{A, p}} \circ \phi_{A,p}^{-1} : \ell_p^n \rightarrow \ell_q^m$. The computations made in the first part of the proof show that $z = T_m(\tilde{v})$ where $m$ is the family $(\phi_{ij})$.\\
Now, let $x:= \varphi_{A,p}(h_1)$ and $y:=\psi_{B,q'}(h_2)$. Since $T_m$ is bounded with norm smaller than $\alpha$ we have
\begin{equation}\label{equa4}
| \left\langle [T_m(\tilde{c})](x),y \right\rangle_{\ell_q^m, \ell_{q'}^m} | \leq \alpha \|\tilde{c}\|_{\mathcal{B}(\ell_p^n,\ell_q^m)} \|x\|_{\ell_p^n} \|y\|_{\ell_{q'}^m}.
\end{equation}
An easy computation shows that the left-hand side on this equality is nothing but the left-hand side of the inequality $(\ref{equa3})$. Finally, the right-hand side of the inequalities $(\ref{equa3})$ and $(\ref{equa4})$ are equal, which concludes the proof.

\end{proof}

\vspace{0.3cm}

\section{$(p,q)-$Factorable operators}

\noindent Let $X$ and $Y$ be Banach spaces.

\subsection{Dual norm.} \cite[Chapter 15]{Defant}. Let $M\subset X$ and $N\subset Y$ be finite dimensional subspaces (in short, f.d.s). If $u = \sum_{i=1}^n x_i\otimes y_i \in M\otimes N$ and $v = \sum_{j=1}^m x_j^* \otimes y_j^* \in M^*\otimes N^*$ we set
$$\left\langle v,u \right\rangle = \sum_{i,j} \left\langle x_j^*,x_i \right\rangle \left\langle y_j^*,y_i \right\rangle.$$

\noindent Let $\alpha$ be a tensor norm on tensor products of finite dimensional spaces. We define, for $z\in M\otimes N$,$$\alpha'(z, M, N)=\sup \left\lbrace |\left\langle v,u \right\rangle| \ | \ v\in M^*\otimes N^*, \alpha(v)\leq 1 \right\rbrace.$$

\noindent Now, for $z\in X\otimes Y$, we set
$$\alpha'(z, X, Y) = \inf \left\lbrace \alpha'(z, M, N) \ | \ M\subset X, N\subset Y \ \text{f.d.s.}, \ z\in M\otimes N \right\rbrace.$$
$\alpha'$ defines a tensor norm on $X\otimes Y$, called the dual norm of $\alpha$. 

In the sequel, we will write $\alpha'(z)$ instead of $\alpha'(z, X, Y)$ for the norm of an element $z \in X \otimes Y$ when there is no possible confusion.

\subsection{Laprest\'{e} norms.} \cite[Proposition 12.5]{Defant}. Let $s \in [1,\infty]$. If $x_1,x_2, \ldots,x_n \in X$, we define
$$w_s(x_i,X):=\underset{x^* \in B_{X^*}}{\sup} \ \left( \sum_{i=1}^n |\left\langle x^*,x_i \right\rangle|^s \right)^{1/s}.$$
\noindent Let $p,q\in [1,\infty]$ with $\dfrac{1}{p}+\dfrac{1}{q}\geq 1$ and take $r\in [1,\infty]$ such that
$$\dfrac{1}{r}=\dfrac{1}{p}+\dfrac{1}{q}-1.$$
\noindent Denote by $p'$ and $q'$ the conjugate of $p$ and $q$. For $z\in X\otimes Y$, we define
$$\alpha_{p,q}(z)=\inf \left\lbrace \|(\lambda_i)_i\|_{\ell_r}w_{q'}(x_i,X)w_{p'}(y_i,Y) \ | \ z=\sum_{i=1}^n \lambda_i x_i\otimes y_i \right\rbrace.$$

\noindent Then $\alpha_{p,q}$ is a norm on $X\otimes Y$ and we denote by $X \otimes_{\alpha_{p,q}} Y$ its completion.\\

\subsection{$(p,q)-$Factorable operators.}

If $T\in \mathcal{B}(X,Y^*)$ and $\xi = \sum_i x_i\otimes y_i \in X\otimes Y$, then in accordance with $(\ref{dualproj})$ we set
$$\left\langle T, \xi \right\rangle = \sum_i \left\langle T(x_i), y_i \right\rangle.$$

\begin{definition}
Let $1 \leq p,q \leq \infty$ such that $\dfrac{1}{p}+\dfrac{1}{q}\geq 1$.
Let $T\in \mathcal{B}(X,Y^*)$. We say that $T \in \mathcal{L}_{p,q}(X,Y^*)$ if there exists a constant $C \geq 0$ such that
\begin{equation}\label{defLpq}
\forall \xi \in X\otimes Y, \ |\left\langle T, \xi \right\rangle| \leq C \alpha_{p,q}'(\xi).
\end{equation}
In this case, we write $L_{p,q}(T) = \inf \left\lbrace C \ | \ C \ \text{satisfying} \ (\ref{defLpq}) \right\rbrace.$\\
Then $(\mathcal{L}_{p,q}(X,Y^*), L_{p,q})$ is a Banach space, called the space of $(p,q)-$Factorable operators.
\end{definition}

\noindent For a general definition of the spaces $\mathcal{L}_{p,q}(X,Y)$ (including the case when the range is not a dual space), see \cite[Chapter 17]{Defant}.\\

\noindent Since $Y^*$ is 1-complemented in its bidual, \cite[Theorem 18.11]{Defant} gives the following result.

\begin{theorem}\label{Facto2}
Let $1 \leq p,q \leq \infty$ such that $\dfrac{1}{p}+\dfrac{1}{q}\geq 1$. Let $T\in \mathcal{B}(X,Y^*)$. The two following statements are equivalent :\\
$(i)$ $T \in \mathcal{L}_{p,q}(X,Y^*)$.\\
$(ii)$ There are a measure space $(\Omega,\mu)$ (a probability space when $\dfrac{1}{p}+\dfrac{1}{q}> 1$), operators $R\in \mathcal{B}(X,L^{q'}(\mu))$ and $S\in \mathcal{B}(L^p(\mu),Y^*))$ such that $T=S\circ I \circ R$

$$\xymatrix{
    X \ar[r]^T \ar[d]_R & Y^* \ar@{<-}[d]^S \\
    L^{q'}(\mu) \ar@{^{(}->}[r]_I & L^p(\mu)
    }$$
where $I : L^{q'}(\mu) \rightarrow L^p(\mu)$ is the inclusion mapping (well defined because $q' \geq p$).\\
\noindent In this case, $L_{p,q}(T)=\inf \|S\| \| R\|$ over all such factorizations.
\end{theorem}

\vspace{0.1cm}

\begin{remark} Here we consider the case when $\dfrac{1}{p}+\dfrac{1}{q}=1$. Denote by $p'$ the conjugate exponent of $p$. We have $T\in \mathcal{L}_{p,p'}(X,Y^*)$ if and only if there are a measure space $(\Omega,\mu)$, operators $R\in \mathcal{B}(X,L^{p}(\mu))$ and $S\in \mathcal{B}(L^p(\mu),Y^*)$ such that $T=SR$

$$\xymatrix @!0 @R=2pc @C=3pc {
X \ar[rr]^T \ar[rdd]_R &  & Y^* \\
 & & \\
& L^p(\mu) \ar[uur]_S &
}$$
We usually write $\Gamma_p(X,Y^*)$ instead of $\mathcal{L}_{p,p'}(X,Y^*)$. Such operators are called $p-$factorable.
\end{remark}

\vspace{0.1cm}

\begin{remark}\label{Facto} Suppose that $X=L^1(\lambda)$ and $Y=L^1(\nu)$ for some localizable measure spaces $(\Omega_1,\lambda)$ and $(\Omega_2,\nu)$. Consider $T \in \mathcal{B}(L^1(\lambda), L^{\infty}(\nu))$. By $(\ref{LinfB})$, there exists $\psi \in L^{\infty}(\lambda \times \nu)$ such that
$$T = u_{\psi}.$$
(See subsection $\ref{Notations}$ for the notation.)
\begin{enumerate}
\item[(i)] If $1 < q < +\infty$, $L^{q'}(\mu)$ has RNP so by ($\ref{Linftensorcor}$),
$$\mathcal{B}(L^1(\lambda), L^{q'}(\mu)) = L^{\infty}(\lambda, L^{q'}(\mu)).$$
It means that if $R\in \mathcal{B}(X,L^{q'}(\mu))$, there exists $a\in L^{\infty}(\lambda, L^{q'}(\mu))$ such that
$$\forall f\in L^1(\lambda), R(f)=\int_{\Omega_1} f(s)a(s)\text{d}\lambda(s).$$
\item[(ii)] If $1<p<+\infty$, then using $(\ref{dualproj})$, $(\ref{L1tensor})$ and $(\ref{L1tensorcor})$ we obtain
\begin{equation*}
B(L^p(\mu), L^{\infty}(\nu)) = (L^p(\mu)\overset{\wedge}{\otimes} L^1(\nu))^* = L^{\infty}(\nu, L^{p'}(\mu)).
\end{equation*}
Thus, if $S\in \mathcal{B}(L^p(\mu),L^{\infty}(\nu))$, there exists $b\in L^{\infty}(\nu, L^{p'}(\mu))$ such that
$$\forall g\in L^p(\lambda), S(g)(\cdot)=\left\langle g, b(\cdot) \right\rangle.$$
\end{enumerate}
We deduce that if $1 < p,q < +\infty$, there exist $a\in L^{\infty}(\lambda, L^{q'}(\mu))$ and $b\in L^{\infty}(\nu, L^{p'}(\mu))$ such that for almost every $(s,t) \in \Omega_1 \times \Omega_2$,
$$\psi(s,t) = \left\langle a(s), b(t) \right\rangle.$$
If $T$ satisfies Theorem $\ref{Facto2}$, the latter implies that for all $f\in L^1(\lambda)$,
$$T(f) = \int_{\Omega_1} \left\langle a(s), b(\cdot) \right\rangle f(s) \ \text{d}s.$$

Using the same identifications we have for the following cases :
\begin{enumerate}
\item If $q=1$ and $1 < p < +\infty$, then there exist $a\in L^{\infty}(\lambda \times \mu)$ and $b\in L^{\infty}(\nu, L^{p'}(\mu))$ such that for almost every $(s,t) \in \Omega_1 \times \Omega_2$,
$$\psi(s,t) = \left\langle a(s,\cdot), b(t) \right\rangle.$$
\item If $1 < q < +\infty$ and $p=+\infty$, then there exist $a\in L^{\infty}(\lambda, L^{q'}(\mu))$ and $b\in L^{\infty}(\nu \times \mu)$ such that for almost every $(s,t) \in \Omega_1 \times \Omega_2$,
$$\psi(s,t) = \left\langle a(s), b(t, \cdot) \right\rangle.$$
\item If $q=1$ and $p=+\infty$, then there exist $a\in L^{\infty}(\lambda \times \mu)$ and $b\in L^{\infty}(\nu \times \mu)$ such that for almost every $(s,t) \in \Omega_1 \times \Omega_2$,
$$\psi(s,t) = \left\langle a(s,\cdot), b(t, \cdot) \right\rangle.$$
\end{enumerate}

\end{remark}

\vspace{0.5cm}

\subsection{Finite dimensional case.} If $X$ and $Y$ are finite dimensional, it follows from the very definition of the dual norm that
$$X \otimes_{\alpha_{p,q}'} Y = (X^* \otimes_{\alpha_{p,q}} Y^*)^*.$$
The next theorem describes the elements of this space.

\begin{theorem}\cite[Theorem 19.2]{Defant} Let $E$ and $F$ be Banach spaces.
Let $p,q \in [1,\infty]$ with $\dfrac{1}{p}+\dfrac{1}{q}\geq 1$ and $K \subset B_{E^*}$ and $L\subset B_{F^*}$ weak$-*$-compact norming sets for E and F, respectively. For $\phi : E\otimes F \rightarrow \mathbb{C}$ the following two statements are equivalent:\\
$(i) \ \phi\in (E\otimes_{\alpha_{p,q}} F)^*$.\\
$(ii)$ There are a constant $A\geq 0$ and normalized Borel-Radon measures $\mu$ on $K$ and $\nu$ on $L$ such that for all $x\in E$ and $y\in F$,
\begin{equation}\label{dominated}
\left\langle \phi, x\otimes y\right\rangle| \leq A \left( \int_K |\left\langle x^*, x \right\rangle|^{q'} \text{d}\mu(x^*) \right)^{1/q'} \left( \int_L |\left\langle y^*, y \right\rangle|^{p'} \text{d}\mu(y^*) \right)^{1/p'}
\end{equation}
(if the exponent is $\infty$, we replace the integral by the norm).\\

\noindent In this case, $\|\phi\|_{(E\otimes_{\alpha_{p,q}} F)^*} = \inf \left\lbrace A \ | \ A \ \text{as in (ii)} \right\rbrace.$
\end{theorem}

\vspace{0.3cm}

\noindent This theorem will allow us to describe the predual of $\mathcal{L}_{p,q}(\ell_1^n,\ell_{\infty}^m)$, $n,m\in \mathbb{N}$. Let us apply the previous theorem with $E=\ell_{\infty}^n$ and $F=\ell_{\infty}^m$.
Take $T\in \ell_1^n \otimes_{\alpha_{p,q}'} \ell_1^m = (\ell_{\infty}^n \otimes_{\alpha_{p,q}} \ell_{\infty}^m)^*$ and let 
$$T= \sum_{i=1}^n \sum_{j=1}^m T(i,j) e_i \otimes e_j$$
be a representation of $T$. In the previous theorem, we can take $K=\left\lbrace 1,2,\ldots,n \right\rbrace$ and $L=\left\lbrace 1,2,\ldots,m \right\rbrace$. In this case, a normalized Borel-Radon measure $\mu$ on $K$ is nothing but a sequence $\mu=(\mu_1,\ldots,\mu_n)$ where, for all $i$, $\mu_i:=\mu(\left\lbrace i \right\rbrace )\geq 0$ and $\sum_i \mu_i=1$. Similarly, $\nu=(\nu_1,\ldots,\nu_m)$ where, for all $i$, $\nu_i\geq 0$ and $\sum_i \nu_i=1$. In this case, the inequality $(\ref{dominated})$ means that for all sequences of complex numbers $x=(x_i)_{i=1}^n, y=(y_j)_{i=j}^m$,
$$\left| \sum_{i=1}^n \sum_{j=1}^m T(i,j) x_i y_j \right| \leq A  \left( \sum_{k=1}^n |x_k|^{q'}\mu_k \right)^{1/q'} \left( \sum_{k=1}^m |y_k|^{p'} \nu_k \right)^{1/p'}.$$
Set $\alpha_k=x_k \mu_k^{1/q'}$, $\beta_k=y_k \nu_k^{1/p'}$ and define, for $1\leq i \leq n, 1\leq j \leq m$, $c(i,j)$ such that $T(i,j)=c(i,j) \mu_i^{1/q'} \nu_j^{1/p'}$ (we can assume $\mu_i>0$ and $\nu_j>0$). Then, the previous inequality becomes
$$\left| \sum_{i=1}^n \sum_{j=1}^m c(i,j) \beta_j \alpha_i \right| \leq A \|\alpha\|_{\ell_{q'}^n}  \|\beta\|_{\ell_{p'}^m}.$$
This means that the operator $c : \ell_{q'}^n \rightarrow \ell_{p}^m$ whose matrix is $[c(i,j)]_{1\leq j \leq m, 1\leq i \leq n}$ has a norm smaller than $A$. Moreover, if we see $T$ as a mapping from $\ell_{\infty}^n$ into $\ell_1^m$ the relation between $T$ and $c$ means that $T$ admits the following factorization

$$\xymatrix{
    \ell_{\infty}^n \ar[r]^T \ar[d]_{d_{\mu}} & \ell_1^m \ar@{<-}[d]^{d_{\nu}} \\
    \ell_{q'}^n \ar[r]_c & \ell_{p}^m
    }$$
where $d_{\mu}$ and $d_{\nu}$ are the operators of multiplication by $\mu=(\mu_1^{1/q'},\ldots,\mu_n^{1/q'})$ and $\nu=(\nu_1^{1/p'},\ldots,\nu_m^{1/p'})$. Those operators have norm 1.\\

\noindent Therefore, it is easy to check that
\begin{equation}\label{formulepredual}
\|T\|_{(\ell_{\infty}^n \otimes_{\alpha_{p,q}} \ell_{\infty}^m)^*} = \inf \left\lbrace  \|c\| \ | \ T=d_{\nu} \circ c \circ d_{\mu}  \right\rbrace.
\end{equation}
The elements of $(\ell_{\infty}^n \otimes_{\alpha_{p,q}} \ell_{\infty}^m)^*$ are called $(q',p')-$dominated operators. For more informations about this space in the infinite dimensional case (it is the predual of $\mathcal{L}_{p,q}$), see for instance \cite[Chapter 19]{Defant}.\\

\noindent By $(\ref{formulepredual})$ and the fact that $\mathcal{L}_{p,q}(\ell_1^n,\ell_{\infty}^n) = (\ell_1^n \otimes_{\alpha_{p,q}'} \ell_1^m)^*$, we get the following result.

\begin{proposition}
Let $v =[v_{ij}] : \ell_1^n \rightarrow \ell_{\infty}^m$. Then
$$L_{p,q}(v)= \sup |\text{Tr}(vu)|$$
where the supremum runs over all $u : \ell_{\infty}^m \rightarrow \ell_1^n$ admitting the factorization
$$\xymatrix{
    \ell_{\infty}^m \ar[r]^u \ar[d]_{d_{\mu}} & \ell_1^n \ar@{<-}[d]^{d_{\nu}} \\
    \ell_{p'}^m \ar[r]_c & \ell_{q}^n
    }$$
with $\|d_{\mu}\| \leq 1, \|d_{\nu}\|\leq 1$ and $\|c\| \leq1$.\\
Equivalently,
$$L_{p,q}(v)=\sup \left\lbrace \left| \sum_{i=1}^m \sum_{j=1}^n v_{ij} c_{ji} \mu_i \nu_j  \right| \ | \ \|c : \ell_{p'}^m \rightarrow \ell_{q}^n\| \leq 1, \| \mu\|_{\ell_{p'}^m} \leq 1, \| \nu\|_{\ell_{q'}^n} \leq 1 \right\rbrace.$$
\end{proposition}

\vspace{0.3cm}

\section{The main result}

\vspace{0.5cm}

\subsection{Schur multipliers and factorization}\label{main}
Let $p,q$ be two positive numbers such that $1\leq q \leq p \leq \infty$. This condition is equivalent to $p,q \in [1,\infty]$ with $\dfrac{1}{q}+\dfrac{1}{p'}\geq 1$, so that we can consider the space $\mathcal{L}_{q,p'}$.\\

The following results will allow us to give a description of the functions $\phi$ which are Schur multipliers.

\vspace{0.3cm}

\begin{lemma}\label{gammapinv}
Let $X$, $Y$ be Banach spaces and let $E \subset X, F\subset Y$ be $1-$complemented subspaces of $X$ and $Y$. For any $v\in E\otimes F$, denote by $\tilde{\alpha}'_{q,p'}(v)$ the $\alpha'_{q,p'}$-norm of $v$ as an element of $E\otimes F$ and by $\alpha'_{q,p'}(v)$ the $\alpha'_{q,p'}$-norm of $v$ as an element of $X\otimes Y$. Then
$$ \tilde{\alpha}'_{q,p'}(v)=\alpha'_{q,p'}(v).$$
\end{lemma}

\begin{proof}
The inequality $ \tilde{\alpha}'_{q,p'}(v)\geq \alpha'_{q,p'}(v)$ is easy to prove. For the converse inequality, take $v=\sum_k e_k\otimes f_k \in E\otimes F$ such that $\alpha'_{q,p'}(v)<1$ and show that $\tilde{\alpha}'_{q,p'}(v) < 1$. By assumption, there exists $M \subset X$ and $N \subset Y$ finite dimensional subspaces such that $v\in M \otimes N$ and
$$\alpha'(v,M,N) < 1.$$
By assumption, there exist two norm one projections $P$ and $Q$ respectively from $X$ onto $E$ and from $Y$ onto $F$. Set $M_1 = P(M) \subset E $ and $N_1 = Q(N) \subset F$. $M_1$ and $N_1$ are finite dimensional. Moreover, since $v \in E \otimes F$, it is easy to check that $(P \otimes Q) (v) =v$, where, for all $c = \sum_l a_l \otimes b_l \in X \otimes Y$,
$$(P\otimes Q)(c)= \sum_l P(a_l) \otimes Q(b_l).$$
Thus, $v \in M_1 \otimes N_1$. We will show that $\alpha'_{q,p'}(v,M_1,N_1) < 1$.\\
Let $z = \sum_{j=1}^m x_j^* \otimes y_j^* \in M_1^*\otimes N_1^*$ be such that $\alpha_{q,p'}(z) < 1$ and show that $|\left\langle v,z \right\rangle | \leq \alpha'_{q,p'}(v)$, so that $\alpha'_{q,p'}(v, M_1, N_1) \leq 1$.\\
Let $1\leq r \leq \infty$ such that $$\dfrac{1}{r}=\dfrac{1}{q}+\dfrac{1}{p'}-1.$$
The condition $\alpha_{q,p'}(z) < 1$ in $M_1^*\otimes N_1^*$ implies that $z$ admits a representation $z = \sum_{j=1}^m \lambda_j m_j^* \otimes n_j^*$ where $m_j^* \in M_1^*, n_j^* \in N_1^*$ and
$$\|(\lambda_j)_j\|_{\ell_r}w_{p}(m_j^*,M_1^*)w_{q'}(n_j^*,N_1^*) < 1.$$
Set $\tilde{z}:= \sum_{j=1}^m \lambda_j P^*(m_j^*) \otimes Q^*(n_j^*)$ in $M^* \otimes N^*$. It is easy to check that
$$w_{p}(P^*(m_j^*),M^*)  \leq w_{p}(m_j^*,M_1^*) \ \ \ \text{and} \ \ \ w_{q'}(Q^*(n_j^*),N^*)  \leq w_{q'}(n_j^*,N_1^*).$$
Therefore, $\alpha_{q,p'}(\tilde{z}, M^*, N^*) < 1$. Then, the condition $\alpha'_{q,p'}(v,M,N) < 1$ implies that
$$|\left\langle v,\tilde{z} \right\rangle | \leq \alpha'_{q,p'}(v).$$
Finally, we have
\begin{align*}
\left\langle v,\tilde{z} \right\rangle
& = \sum_{j,k} \lambda_j \left\langle  P^*(m_j^*),e_k \right\rangle \left\langle Q^*(n_j^*),f_k \right\rangle \\
& = \sum_{j,k} \lambda_j \left\langle  m_j^*,P(e_k) \right\rangle \left\langle n_j^*,Q(f_k) \right\rangle \\
& = \sum_{j,k} \lambda_j \left\langle  m_j^*,e_k \right\rangle \left\langle n_j^*,f_k \right\rangle = \left\langle v,z \right\rangle,
\end{align*}
and therefore
$$|\left\langle v,z \right\rangle | \leq \alpha'_{q,p'}(v).$$
This proves that $\tilde{\alpha}'_{q,p'}(v) < 1$.
\end{proof}

\noindent We recall that if $\phi\in L^{\infty}(\Omega_1 \times \Omega_2)$, we denote by $u_{\phi}$ the mapping
\begin{equation*}
\begin{array}[t]{lrcl}
u_{\phi} : & L^1(\Omega_1) & \longrightarrow & L^{\infty}(\Omega_2). \\
& f & \longmapsto & \displaystyle \int_{\Omega_1} \phi(s,\cdot)f(s)\ \text{d}\mu_1(s) \end{array}
\end{equation*}

\begin{theorem}\label{multipliers}
Let $(\Omega_1,\mu_1)$ and $(\Omega_2, \mu_2)$ be two localizable measure spaces and let $\phi\in L^{\infty}(\Omega_1 \times \Omega_2)$. Let $1\leq q \leq p \leq \infty$. Then $\phi$ is a Schur multiplier on $\mathcal{B}(L^p (\Omega_1), L^q (\Omega_2))$ if and only if the operator $u_{\phi}$ belongs to $\mathcal{L}_{q,p'}(L^1 (\Omega_1), L^{\infty}(\Omega_2))$. Moreover,
$$\| T_{\phi} \| = L_{q,p'}(u_{\phi}).$$
\end{theorem}

\begin{proof}
Assume first that $T_{\phi}$ extends to a bounded operator 
$$T_{\phi} : L^{p'}(\Omega_1) \overset{\vee}{\otimes} L^q(\Omega_2) \rightarrow L^{p'}(\Omega_1) \overset{\vee}{\otimes} L^q(\Omega_2)$$
with norm $\leq 1$. To prove that $u_{\phi} \in \mathcal{L}_{q,p'} (L^1 (\Omega_1), L^{\infty}(\Omega_2))$ with $L_{q,p'}(u_{\phi})\leq 1$, we have to show that for any $v=\sum_k f_k \otimes g_k \in L^1(\Omega_1) \otimes L^1(\Omega_2)$ with $\alpha'_{q,p'}(v) <1$ we have
$$|u_{\phi}(v)|=|\sum_k \left\langle u_{\phi} (f_k),g_k \right\rangle | \leq 1.$$
By density, we can assume that $f_k, g_k$ are simple functions. Hence, with the notations introduced in Section $\ref{first}$ there exist $n,m \in \mathbb{N}^*, A=(A_1,\ldots,A_n) \in \mathcal{A}_{n,\Omega_1}$ and $B=(B_1,\ldots,B_m) \in \mathcal{A}_{m,\Omega_2}$ such that, for all $k$, $f_k \in S_{A,1}$ and $g_k \in S_{B,1}$.\\
By Lemma $\ref{gammapinv}$, the $\alpha'_{q,p'}$-norm of $v$ as an element of $S_{A,1} \otimes S_{B,1}$ is less than 1.\\

Let $\varphi_{A,1} : S_{A, 1} \rightarrow \ell_1^n$ and $\psi_{B,1} : S_{B, 1} \rightarrow \ell_1^m$ the isomorphisms defined in $(\ref{isolp})$. Set $v'=\sum_k \varphi_{A,1}(f_k) \otimes \psi_{B,1}(g_k) \in \ell_1^n \otimes \ell_1^m$. Since $\varphi_{A,1}$ and $\psi_{B,1}$ are isometries, we have $\alpha'_{q,p'}(v') <1$. Using the identification $(\ref{tensorop})$, we obtain by $(\ref{formulepredual})$ that  $v'$ admits a factorization

$$\xymatrix{
    \ell_{\infty}^n \ar[r]^{v'} \ar[d]_{d_{\delta}} & \ell_1^m \ar@{<-}[d]^{d_{\gamma}} \\
    \ell_{p}^n \ar[r]_c & \ell_{q}^m
    }$$
where $\delta=(\delta_1,\ldots,\delta_n)$, $\gamma=(\gamma_1,\ldots,\gamma_m)$, $d_{\delta}$ and $d_{\gamma}$ are the operators of multiplication and
$$\|d_{\delta}\| = \|\delta\|_{\ell_{p}} = 1, \|d_{\gamma}\| = \|\gamma\|_{\ell_{q'}} = 1 \ \text{and} \|c\| < 1.$$
This factorization means that
$$v'=\sum_{i=1}^m \sum_{j=1}^n \gamma_i  c(i,j) \delta_j e_j \otimes e_i.$$

\noindent Therefore, we have
\begin{align*}
v
& =\sum_{i=1}^m \sum_{j=1}^n \gamma_i  c(i,j) \delta_j \ \varphi_{A,1}^{-1}(e_j) \otimes \psi_{B,1}^{-1}(e_i) \\
& = \sum_{i=1}^m \sum_{j=1}^n \gamma_i \dfrac{c(i,j)}{\mu_1(A_j) \mu_2(B_i)}  \delta_j \ \chi_{A_j} \otimes \chi_{B_i}.
\end{align*}
We compute
\begin{align*}
u_{\phi}(v)
& = \sum_{i=1}^m \sum_{j=1}^n \gamma_i \dfrac{c(i,j)}{\mu_1(A_j) \mu_2(B_i)}  \delta_j \left\langle u_{\phi}(\chi_{A_j}), \chi_{B_i} \right\rangle \\
& = \sum_{i=1}^m \sum_{j=1}^n \gamma_i \dfrac{c(i,j)}{\mu_1(A_j) \mu_2(B_i)}  \delta_j \left\langle T_{\phi}(\chi_{A_j} \otimes\chi_{B_i})(\chi_{A_j}), \chi_{B_i} \right\rangle
\end{align*}
Define
$$\tilde{c} = \sum_{i=1}^m \sum_{j=1}^n \tilde{c}(i,j) \chi_{A_j} \otimes\chi_{B_i} \in L^{p'}(\Omega_1) \otimes L^q(\Omega_2),$$
where $\tilde{c}(i,j)=c_{i,j} \mu_1(A_j)^{-1/p'} \mu_2(B_i)^{-1/q}$.\\
Using the identification $(\ref{tensorop})$, it is easy to check that we have
$$\tilde{c} = \psi_{B,q}^{-1} \circ c \circ \varphi_{A,p} : S_{A,p} \mapsto L^q(\Omega_2).$$
Therefore,
$$\|\tilde{c}\|_{\vee} = \|c\|.$$

We have
\begin{align*}
u_{\phi}(v)
& = \sum_{i=1}^m \sum_{j=1}^n \gamma_i \dfrac{\tilde{c}(i,j)\mu_1(A_j)^{1/p'} \mu_2(B_i)^{1/q}}{\mu_1(A_j) \mu_2(B_i)}  \delta_j \left\langle T_{\phi}(\chi_{A_j} \otimes\chi_{B_i})(\chi_{A_j}), \chi_{B_i} \right\rangle \\
& = \sum_{i=1}^m \sum_{j=1}^n \gamma_i \tilde{c}(i,j)\mu_1(A_i)^{-j1/p} \mu_2(B_i)^{-1/q'} \delta_j \left\langle T_{\phi}(\chi_{A_j} \otimes\chi_{B_i})(\chi_{A_j}), \chi_{B_i} \right\rangle \\
& = \sum_{i=1}^m \sum_{j=1}^n \left\langle T_{\phi}(\tilde{c}(i,j) \chi_{A_j} \otimes\chi_{B_i})\left(\dfrac{\delta_j}{\mu_1(A_j)^{1/p}}\chi_{A_j}\right), \dfrac{\gamma_i}{\mu_2(B_i)^{1/q'}} \chi_{B_i} \right\rangle \\
& = \left\langle T_{\phi}(\tilde{c})(f), g \right\rangle_{L^q(\Omega_2), L^{q'}(\Omega_2)},
\end{align*}
where
$$f = \sum_j \dfrac{\delta_j}{\mu_1(A_j)^{1/p}}\chi_{A_j} \ \ \text{and} \ \ g=\sum_i \dfrac{\gamma_i}{\mu_2(B_i)^{1/q'}} \chi_{B_i}.$$
Since $\|T_{\phi}\| \leq 1$, we deduce that
\begin{align*}
|u_{\phi}(v)| \leq \|T_{\phi}(\tilde{c})\| \|f\|_p \|g\|_{q'} \leq \|\tilde{c}\| \|\delta\|_{\ell_p} \|\gamma\|_{\ell_{q'}} = \|c\| \leq 1.
\end{align*}

\vspace{0.3cm}

Conversely, assume that $u_{\phi} \in \mathcal{L}_{q,p'}(L^1 (\Omega_1), L^{\infty}(\Omega_2))$ with $L_{q,p'}(u_{\phi}) \leq 1$. To prove that $\phi$ is a Schur multiplier, we will use Proposition $\ref{discret}$. Let $n,m \in \mathbb{N}^*$, $A=(A_1,\ldots,A_n) \in \mathcal{A}_{n,\Omega_1}$ and $B=(B_1,\ldots,B_m) \in \mathcal{A}_{m,\Omega_2}$. Set
$$\phi_{ij} = \dfrac{1}{\mu_1(A_j)\mu_2(B_i)} \displaystyle \int_{A_j \times B_i} \phi \ \text{d}\mu_1 \text{d}\mu_2.$$
We want to show that the Schur multiplier on $\mathcal{B}(\ell_p^n,\ell_q^m)$ associated to the family $m = (\phi_{ij})_{i,j}$ has a norm less than $1$. To prove that, let $c=\sum_{i,j} c(i,j) e_j \otimes e_i \in \mathcal{B}(\ell_p^n,\ell_q^m), x=(x_j)_{j=1}^n, y=(y_i)_{i=1}^m$ in $\mathbb{C}$ be such that $\|c\| \leq 1, \|x\|_{\ell_p^n} =1, \|y\|_{\ell_{q'}} = 1$. We have to show that
$$| \left\langle [T_m(c)](x),y \right\rangle_{\ell_q^m, \ell_{q'}^m} | \leq 1.$$ 
This inequality can be rewritten as
\begin{equation}\label{proof1}
\left| \sum_{i,j} c(i,j) \dfrac{x_j y_i}{\mu_1(A_j)\mu_2(B_i)} \left( \int_{A_j \times B_i} \phi \right) \right| \leq 1.
\end{equation}
Let $v = \sum_{i,j} x_j c(i,j) y_i e_j \otimes e_i$. According to $(\ref{formulepredual})$, $\alpha'_{q,p'}(v) \leq 1$. Now, let $\tilde{v} = \sum_{i,j} x_j c(i,j) y_i \varphi_{A,1}^{-1}(e_j) \otimes \psi_{B,1}^{-1}(e_i)$. We have
$$\alpha'_{q,p'}(\tilde{v}) = \alpha'_{q,p'}(v) \leq 1$$
and
$$\tilde{v} = \sum_{i,j} \dfrac{x_j c(i,j) y_i}{\mu_1(A_j)\mu_2(B_i)} \chi_{A_j} \otimes \chi_{B_i}.$$
By assumption, $L_{q,p'}(u_{\phi}) \leq 1$, which implies that
\begin{align*}
| \left\langle u_{\phi}, \tilde{v} \right\rangle |
& =  \left| \sum_{i,j} c(i,j) \dfrac{x_j y_i}{\mu_1(A_j)\mu_2(B_i)} \left( \int_{A_j \times B_i} \phi \right) \right| \\
& \leq \alpha'_{q,p'}(\tilde{v}) \\
& \leq 1,
\end{align*}
and this is precisely the inequality $(\ref{proof1})$.
\end{proof}

\vspace{0.5cm}

\noindent Theorem $\ref{Facto2}$ and Remark $\ref{Facto}$ allow us to reformulate the previous theorem. The following two corollaries are generalizations of Theorem $\ref{FactoPisier}$.

\begin{corollary}\label{Schurfacto}
Let $(\Omega_1,\mu_1)$ and $(\Omega_2, \mu_2)$ be two localizable measure spaces and let $\phi\in L^{\infty}(\Omega_1 \times \Omega_2)$. Let $1\leq q \leq p \leq \infty$. The following statements are equivalent :
\begin{enumerate}
\item[(i)] $\phi$ is a Schur multiplier on $\mathcal{B}(L^p (\Omega_1), L^q (\Omega_2))$.\
\item[(ii)] There are a measure space (a probability space when $p\neq q$) $(\Omega,\mu)$, operators $R\in \mathcal{B}(L^1(\Omega_1),L^{p}(\mu))$ and $S\in \mathcal{B}(L^q(\mu),L^{\infty}(\Omega_2))$ such that $u_{\phi}=S\circ I \circ R$

$$\xymatrix{
    L^1(\Omega_1) \ar[r]^{u_{\phi}} \ar[d]_R & L^{\infty}(\Omega_2) \ar@{<-}[d]^S \\
    L^{p}(\mu) \ar@{^{(}->}[r]_I & L^q(\mu)
    }$$
    where $I$ is the inclusion mapping.
\end{enumerate}
In the following cases, $(i)$ and $(ii)$ are equivalent to :\\
If $1<q\leq p<+\infty$ :
\begin{enumerate}
\item[(iii)] There are a measure space (a probability space when $p\neq q$) $(\Omega,\mu)$, $a\in L^{\infty}(\mu_1, L^{p}(\mu))$ and $b\in L^{\infty}(\mu_2, L^{q'}(\mu))$ such that, for almost every $(s,t) \in \Omega_1  \times \Omega_2$,
$$\phi(s,t)=\left\langle a(s), b(t) \right\rangle.$$
\end{enumerate}
If $1=q < p<+\infty$ :
\begin{enumerate}
\item[(iii)] There are a probability space $(\Omega,\mu)$, $a\in L^{\infty}(\mu_1 \times \mu)$ and $b\in L^{\infty}(\mu_2, L^{q'}(\mu))$ such that for almost every $(s,t) \in \Omega_1 \times \Omega_2$,
$$\phi(s,t) = \left\langle a(s,\cdot), b(t) \right\rangle.$$
\end{enumerate}
If $1 < q < +\infty$ and $p=+\infty$ :
\begin{enumerate}
\item[(iii)] There are a probability space $(\Omega,\mu)$, $a\in L^{\infty}(\mu_1, L^{p}(\mu))$ and $b\in L^{\infty}(\mu_2 \times \mu)$ such that for almost every $(s,t) \in \Omega_1 \times \Omega_2$,
$$\phi(s,t) = \left\langle a(s), b(t, \cdot) \right\rangle.$$
\end{enumerate}
If $q=1$ and $p=+\infty$ :
\begin{enumerate}
\item[(iii)] There are a probability space $(\Omega,\mu)$, $a\in L^{\infty}(\mu_1 \times \mu)$ and $b\in L^{\infty}(\mu_2 \times \mu)$ such that for almost every $(s,t) \in \Omega_1 \times \Omega_2$,
$$\phi(s,t) = \left\langle a(s,\cdot), b(t, \cdot) \right\rangle.$$
\end{enumerate}
\vspace{0.2cm}
In this case, $\|T_{\phi}\| = \inf \|R\| \|I\| \|S\| = \inf \|a\| \|b\|$.
\end{corollary}

\vspace{0.3cm}

\begin{remark}
In the previous corollary, the condition $(ii)$ implies that every $\phi \in L^{\infty}(\Omega_1 \times \Omega_2)$ is a Schur multiplier on $\mathcal{B}(L^1(\Omega_1), L^1(\Omega_2))$ and on $\mathcal{B}(L^{\infty}(\Omega_1), L^{\infty}(\Omega_2))$.
\end{remark}

\vspace{0.3cm}

\noindent In the discrete case, the previous corollary can be reformulated as follow.

\begin{corollary}\label{Schurfacto2}
Let $\phi=(c_{ij})_{i,j\in \mathbb{N}} \subset \mathbb{C}$, $C\geq 0$ be a constant and let $1 \leq q \leq p \leq +\infty$. The following are equivalent :
\begin{enumerate}
\item[(i)] $\phi$ is a Schur multiplier on $\mathcal{B}(\ell_p,\ell_q)$ with norm $\leq C$.
\item[(ii)] There exist a measure space (a probability space when $p\neq q$) $(\Omega, \mu)$ and two bounded sequences $(x_j)_j$ in $L^p(\mu)$ and $(y_i)_i$ in $L^{q'}(\mu)$ such that
$$\forall i,j \in \mathbb{N}, c_{ij} = \left\langle x_j, y_i \right\rangle \ \text{and} \ \sup_i \|y_i\|_{q'} \sup_j \|x_j\|_p \leq C.$$
\end{enumerate}
\end{corollary}

\vspace{0.2cm}

\subsection{An application : the main triangle projection}
Let $m_{ij} = 1$ if $i\leq j$ and $m_{ij} = 0$ otherwise. Let $T_m$ be the Schur multiplier associated with the family $m=(m_{ij})$. For any infinite matrix $A=[a_{ij}]$, $T_m(A)$ is the matrix $[b_{ij}]$ with $b_{ij}=a_{ij}$ if $i\leq j$ and $b_{ij}=0$ otherwise. For that reason, $T_m$ is called the main triangle projection. Similary, we define the \textit{n}-th main triangle projection as the Schur multiplier on $\mathcal{M}_n(\mathbb{C})$ associated with the family $m_n = (m_{ij}^n)_{1 \leq i,j \leq n}$ where $m_{ij}^n = 1$ if $i\leq j$ and $m_{ij}^n = 0$ otherwise.
In \cite{Kwapien}, Kwapie\'n and Pelczy\'nski proved that if $1\leq q \leq p \leq +\infty, p\neq 1, q\neq +\infty$, there exists a constant $K > 0$ such that for all $n$,
$$\| T_{m_n} : \mathcal{B}(\ell_p^n, \ell_q^n) \rightarrow \mathcal{B}(\ell_p^n, \ell_q^n)\| \geq K \ln (n),$$
and this order of growth is obtained for the Hilbert matrices. Those estimates imply that $T_m$ is not bounded on $\mathcal{B}(\ell_p, \ell_q)$.
Bennett proved in \cite{Bennett2} that when $1<p<q<\infty$, $T_m$ is bounded from $\mathcal{B}(\ell_p, \ell_q)$ into itself.

The results obtained in subsection $\ref{main}$ allow us to give a very short proof of the unbounded case.

\begin{proposition}
Let $1 \leq q \leq p \leq +\infty,p\neq 1, q\neq +\infty$. Then $T_m$ is not bounded on $\mathcal{B}(\ell_p, \ell_q)$.
\end{proposition}

\begin{proof}
Assume that $T_m$ is bounded on $\mathcal{B}(\ell_p, \ell_q)$. By Corollary $\ref{Schurfacto}$, there exist a measure space $(\Omega, \mu)$, $(a_n)_n \in L^p(\mu)$ and $(b_n)_n \in L^{q'}(\mu)$ two bounded sequences such that, for all $i,j \in \mathbb{N}$,
\begin{equation}\label{factotriangle}
m_{ij}=\left\langle a_j, b_i \right\rangle.
\end{equation}
By boundedness, $(a_n)_n$ and $(b_n)_n$ admit an accumluation point $a\in L^p(\mu)$ and $b\in L^{q'}(\mu)$ respectively for the weak-* topology. Fix $i \in \mathbb{N}$. For all $j\geq i$, we have
$$\left\langle a_i, b_j \right\rangle  = 1$$
so that we get
$$\left\langle a_i ,b \right\rangle = 1.$$
This equality holds for any $i$ hence
$$\left\langle a ,b \right\rangle = 1.$$
Now fix $j\in \mathbb{N}$. For all $i>j$ we have
$$\left\langle a_i, b_j \right\rangle = 0.$$
From this, we deduce as above that
$$\left\langle a, b \right\rangle = 0.$$
We obtained a contradiction so $T_m$ cannot be bounded.
\end{proof}

\noindent As a consequence, we have, by Proposition $\ref{discret}$ :

\begin{corollary} Let $1 \leq q \leq p \leq +\infty,p\neq 1, q\neq +\infty$.
Let $\Omega_1 = \Omega_2 = \mathbb{R}$ with the Lebesgue measure. Then $\phi \in L^{\infty}(\mathbb{R}^2)$ defined by
\begin{align*}
\phi(s,t) :=
\begin{cases}1, & \text{if~$s+t \geq 0$} \\
0 & \text{if~$s+t <0$}
\end{cases}, \ \ s, t \in\mathbb{R} 
\end{align*}
is not a Schur multiplier on $\mathcal{B}(L^p(\mathbb{R}), L^q(\mathbb{R}))$.
\end{corollary}

\begin{remark}
One could wonder whether the results of subsection $\ref{main}$ can be extended to the case $1\leq p < q \leq +\infty$, that is, if the boundedness of $T_{\phi}$ on $\mathcal{B}(L^p, L^q)$ implies that $u_{\phi}$ has a certain factorization. The fact that if $p<q$ the main triangle projection is bounded tells us that $m$ is a Schur multiplier on $\mathcal{B}(\ell_p, \ell_q)$. Nevertheless, the argument used in the previous proof shows that $m$ cannot have a factorization like in $(\ref{factotriangle})$. Therefore, the case $p<q$ is more tricky. For the discrete case, one can find in \cite[Theorem 4.3]{Bennett} a necessary and sufficient condition for a family $(m_{i,j}) \subset \mathbb{C}$ to be a Schur multiplier, for all values of $p$ and $q$, using the theory of $q-$absolutely summing operators.
\end{remark}

\vspace{0.3cm}


\section{Inclusion theorems}

\noindent In this section, we denote by $\mathcal{M}(p,q)$ the space of Schur multipliers on $\mathcal{B}(\ell_p,\ell_q)$.\\

\noindent First, we recall the inclusion relationships between the spaces $\mathcal{M}(p,q)$. Then we will establish new results as applications of those obtained in Section $\ref{main}$.

\begin{theorem}\cite[Theorem 6.1]{Bennett}\label{incluB}
Let $p_1 \geq p_2$ and $q_1 \leq q_2$ be given. Then $\mathcal{M}(p_1,q_1) \subset \mathcal{M}(p_2,q_2)$ with equality in the following cases:
\begin{enumerate}
\item[(i)] $p_1 = p_2 = 1$,
\item[(ii)] $q_1=q_2=\infty$,
\item[(iii)] $q_2 \leq 2 \leq p_2$,
\item[(iv)] $q_2 < p_1 = p_2 < 2$,
\item[(v)] $2 < q_1 = q_2 < p_2$.
\end{enumerate}
\end{theorem}

Let $(\Omega_1,\mu_1)$ and $(\Omega_2, \mu_2)$ be two measure spaces. If $\mathcal{M}(p_1,q_1) \subset \mathcal{M}(p_2,q_2)$, then using Proposition $\ref{discret}$ we have that any Schur multiplier on $\mathcal{B}(L^{p_1}(\Omega_1), L^{q_1}(\Omega_2))$ is a Schur multiplier on $\mathcal{B}(L^{p_2}(\Omega_1), L^{q_2}(\Omega_2))$. Hence, the results in the previous theorem hold true for all the Schur multipliers on $\mathcal{B}(L^p,L^q)$.
\vspace{0.2cm}

In the sequel, we will need the notion of type for a Banach space $X$, for which we refer e.g. to \cite{Kalton}. Let $(\mathcal{E}_i)_{i\in \mathbb{N}}$ be a sequence of independent Rademacher random variables. We have the following definition.

\begin{definition} A Banach space X is said to have Rademacher type p (in short, type p) for some $1 \leq p \leq 2$ if there is a constant $C$ such that for every finite set of vectors $(x_i)_{i=n}^n$ in $X$,
\begin{equation}\label{type}
\left( \mathbb{E} \left\| \sum_{i=1}^n\mathcal{E}_i x_i \right\|^p \right)^{1/p} \leq C \left( \sum_{i=1}^n \|x_i\|^p \right)^{1/p}.
\end{equation}
The smallest constant $C$ for which $(\ref{type})$ holds is called the type-p constant of $X$.
\end{definition}

We will use the fact that for $1 \leq p \leq 2$, $L^p$-spaces have type $p$ and if $2 < p < +\infty$, $L^p$-spaces have type 2 and that those are the best types for infinite dimensional $L^p$-spaces (see for instance \cite[Theorem 6.2.14]{Kalton}). We will also use the fact that the type is stable by passing to quotients. Namely, if $X$ has type $p$ and $E\subset X$ is a closed subspace, then $X/E$ has type $p$.

\begin{proposition}\label{inclusions}
$(i)$ If \ $1\leq q < p \leq 2$, then
$$\mathcal{M}(q,1) \nsubseteq \mathcal{M}(p,p).$$
Consequently, for any $1 \leq r \leq q$,
$$\mathcal{M}(q,r) \nsubseteq \mathcal{M}(p,p).$$
$(ii)$ If \ $2 \leq p < q \leq r$, then
$$\mathcal{M}(r,q) \nsubseteq \mathcal{M}(p,p).$$
$(iii)$ If $1 < q < 2 < p < +\infty$ or $1 < p < 2 < q < +\infty$, then
$$\mathcal{M}(q,q) \nsubseteq \mathcal{M}(p,p).$$
\end{proposition}

\vspace{0.2cm}

\noindent To prove this proposition, we will need the following definitions and lemma.

\begin{definition}
Let $X$ and $Y$ be Banach spaces. A map $s : X \rightarrow Y$ is a quotient map if $s$ is surjective and for all $y \in Y$ with $\|y\| < 1$, there exists $x\in X$ such that $\|x\| < 1$ and $s(x)=y$. This is equivalent to the fact that the injective map $\hat{s} : X/\ker(s) \rightarrow Y$ induced by $s$ is a surjective isometry.
\end{definition}

\begin{definition}
Let $X$ and $Y$ be Banach spaces, $u\in \mathcal{B}(X,Y)$ and $1 \leq p \leq \infty$. We say that $u \in SQ_p(X,Y)$ if there exists a closed subspace $Z$ of a quotient of a $L^p$-space and two operators $A\in \mathcal{B}(X,Z)$ and $B\in \mathcal{B}(Z,Y)$ such that $u= BA$.\\
Then $\| u \|_{SQ_p} = \inf \|A\| \|B\|$ defines a norm on $SQ_p(X,Y)$ and $(SQ_p(X,Y), \|.\|_{SQ_p)}$ is a Banach space.
\end{definition}

\begin{lemma}\label{SQ}
Let $W, X, Y, Z$ be Banach spaces and let $u\in \mathcal{B}(X,Y), s\in \mathcal{B}(W,X), v\in \mathcal{B}(Y,Z)$ such that
$s$ is a quotient map, $v$ is a linear isometry and $vus \in \Gamma_p(W,Z)$. Then $u \in SQ_p(X,Y)$.
\end{lemma}

\begin{proof}
By assumption, there exist a $L^p$-space $U$ and two operators $a\in \mathcal{B}(W,U)$ and $b\in \mathcal{B}(U,Z)$ such that the following diagram commutes

$$\xymatrix @!0 @R=2pc @C=2pc {
W \ar@{->>}[rr]^s \ar[rrrdd]_{a} && X  \ar[rr]^u && Y \ar@{^(->}[rr]^v && Z   \\
& & & \\
&&& U \ar[rrruu]_b
}$$
Since $v$ is an isometry, $V := v(Y) \subset Z$ is isometrically isomorphic to $Y$. Let $\psi : Y \rightarrow V$ be the isometric isomorphism induced by $v$.\\
Set $F := \left\lbrace x \in U \ \text{such that} \ b(x) \in V \right\rbrace.$ Since $vus=ba$, we have, for all $w\in W, v(us(w))=b(a(w))$, so that $a(w) \in F$. This implies that $a(W) \subset F$. We still denote by $a$ the mapping $a : W \rightarrow F$ and by $b$ the restriction of $b$ to $F$. Denote by $\hat{b}$ the mapping $\hat{b}=\psi^{-1} \circ b : F \rightarrow Y$. Then we have the following commutative diagram

$$\xymatrix{
W \ar@{->>}[r]^s \ar[rd]_{a}
 &
X  \ar[r]^u & Y  \\
& F \ar[ur]_{\hat{b}}
}$$
Now, set $E:=\overline{a(\ker(s))}$ and let $Q : F \rightarrow F/E$ be the canonical mapping. Clearly, $Q \circ a : W \rightarrow F/E$ vanishes on $\ker(s)$, so that we have a mapping
$$ \widehat{Q \circ a} : W/\ker(s) \rightarrow F/E$$
induced by $Q \circ a$.\\
Since $s$ is a quotient map, we denote by $\widehat{s}$ the isometric isomorphism
$$\hat{s} : W/\ker(s) \rightarrow X.$$
Define
$$A = \widehat{Q \circ a} \circ \hat{s}^{-1} : X \rightarrow F/ E.$$
$\hat{b}$ vanishes on $E$ so that we have a mapping
$$B : F/E \rightarrow Y.$$
Finally, it is easy to check that $u = BA$, that is, we have the following commutative diagram
$$\xymatrix @!0 @R=2pc @C=3pc {
X  \ar[rr]^u \ar[rdd]_{A} &  &Y \\
 & & \\
& F/E \ar[uur]_B &
}$$
which concludes the proof.
\end{proof}

\begin{remark}
To prove Lemma $\ref{SQ}$, one can use a result of Kwapie\'n characterizing elements of $SQ_p$, as follows : a Banach space $X$ is isomorphic to an $SQ_q$-space if and only if there exists a constant $K\geq 1$ such that for any $n\geq 1$, for any $n\times n$ matrix $[a_{ij}]$ and for any $x_1,\ldots,x_n$ in $X$,
$$\left( \sum_i \left\| \sum_j a_{ij}x_j \right\|^q \right)^{1/q} \leq K \|[a_{ij}] : \ell_q^n \rightarrow \ell_q^n \| \left( \sum_j \|x_j\|^q \right)^{1/q}.$$
\noindent However, the proof presented in this paper also works if we replace in the statement of the lemma $\Gamma_p$ (respectively $SQ_p$) by the space of operators that can be factorized by some Banach space $L$ (respectively by a subspace of a quotient of $L$).
\end{remark}

\begin{proof}[Proof of Proposition $\ref{inclusions}$] $(i)$.
Let $\Omega:=[0,1]$ and $\lambda$ be the Lebesgue measure on $\Omega$. Let $I_q : L^q(\lambda) \rightarrow L^1(\lambda)$ be the inclusion mapping. By the classical Banach space theory (see \cite[Theorem 2.3.1]{Kalton} and \cite[Theorem 2.5.7]{Kalton}) there exist a quotient map $ \sigma : \ell_1 \twoheadrightarrow L^q(\lambda)$ and an isometry $J : L^1(\lambda) \hookrightarrow \ell_{\infty}$. Let $\phi\in \ell_{\infty}(\mathbb{N}^2)$ be such that
$$u_\phi = J I_q \sigma$$
(by $(\ref{LinfB})$ any continuous linear map $\ell_1 \rightarrow \ell_{\infty}$ is a certain $u_{\phi}$ for $\phi \in L^{\infty}(\mathbb{N}\times \mathbb{N})$). We have the following factorization
$$\xymatrix{
    \ell_1 \ar[r]^{u_{\phi}} \ar[d]_{\sigma} & \ell_{\infty} \ar@{<-}[d]^{J} \\
    L^{q}(\lambda) \ar@{^{(}->}[r]_{I_q} & L^1(\lambda)
    }$$
According to Theorem $\ref{Schurfacto}$, $\phi \in \mathcal{M}(q,1)$.\\

Assume that $\phi\in \mathcal{M}(p,p)$. Then, again by Theorem $\ref{Schurfacto}$, we have $u_{\phi} \in \Gamma_p(\ell_1,\ell_{\infty})$ and therefore, by Lemma $\ref{SQ}$, there exist an $SQ_p$-space $X$ and two operators $\alpha \in \mathcal{B}(L^q(\lambda), X)$ and $\beta\in\mathcal{B}(X,L^1(\lambda))$ such that $I_q = \beta \alpha$.\\
Let $(\mathcal{E}_i)_{i\in \mathbb{N}}$ be a sequence of independant Rademacher random variables. Let $n\in \mathbb{N}^*$ and $f_1,\ldots, f_n \in L^q(\lambda)$.
\begin{align*}
\mathbb{E} \left\| \sum_{j=1}^n \mathcal{E}_j f_j \right\|_{L^1(\lambda)}
= \mathbb{E} \left\| \sum_{j=1}^n \mathcal{E}_j \beta \alpha(f_j) \right\|_{L^1(\lambda)}
\leq \| \beta\| \mathbb{E} \left\| \sum_{j=1}^n \mathcal{E}_j \alpha(f_j) \right\|_X.
\end{align*}
But $X$ has type $p$ so there exists a constant $C_1 > 0$ such that
\begin{align*}
\mathbb{E} \left\| \sum_{j=1}^n \mathcal{E}_j f_j \right\|_{L^1(\lambda)}
\leq C_1 \| \beta\| \left( \sum_{j=1}^n \| \alpha(f_j) \|^p_X \right)^{1/p}
\leq C_1 \|\beta\| \|\alpha\| \left( \sum_{j=1}^n \| f_j \|^p_{L^q(\lambda)} \right)^{1/p}.
\end{align*}
By Khintchine inequality, there exists $C_2 > 0$ such that
$$\left\| \left( \sum_{j=1} |f_j|^2 \right)^{1/2} \right\|_{L^1(\lambda)} \leq C_2 \mathbb{E} \left\| \sum_{j=1}^n \mathcal{E}_j f_j \right\|_{L^1(\lambda)}.$$
Thus, setting $K := C_1 C_2 \| \alpha \| \| \beta\|$, we obtained the inequality
$$\left\| \left( \sum_{j=1} |f_j|^2 \right)^{1/2} \right\|_{L^1(\lambda)} \leq K \left( \sum_{j=1}^n \| f_j \|^p_{L^q(\lambda)} \right)^{1/p}.$$
Let $E_1, \ldots, E_n$ be disjoint measurable subsets of $[0,1]$ such that for all $1 \leq j \leq n, \lambda(E_j)=\dfrac{1}{n}$. Set $f_j  := \chi_{E_j}$. Then
$$\sum_j |f_j|^2 = 1 \ \ \ \text{and} \ \ \ \|f\|_{L^q(\lambda)} = n^{-1/q}.$$
Hence, applying the previous inequality to the $f_j$'s, we obtain
$$1 \leq K n^{1/p-1/q}.$$
Since $q < p$, this inequality can't hold for all $n$, so we obtained a contradiction.\\

\noindent Finally, notice that if $1 \leq r \leq q$, then by Theorem $\ref{incluB}$, $\mathcal{M}(q,1) \subset \mathcal{M}(q,r)$. Thus, $\mathcal{M}(q,r) \nsubseteq \mathcal{M}(p,p).$\\

$(ii)$.
By Proposition $\ref{discret}$ and using duality, it is easy to prove that for all $s,t \in [1,\infty], \phi$ is a Schur multiplier on $\mathcal{B}(\ell_s,\ell_t)$ if and only if $\tilde{\phi}$ is a Schur multiplier on $\mathcal{B}(\ell_{t'}, \ell_{s'})$, where $\tilde{\phi}$ is defined for all $i,j\in \mathbb{N}$ by $\tilde{\phi}(i,j)=\phi(j,i)$.\\
Let $2 \leq p < q \leq r$. Then $1 \leq r' \leq q' < p' \leq 2$. If we assume that $\mathcal{M}(r,q) \subset \mathcal{M}(p,p)$ then the latter implies $\mathcal{M}(q',r') \subset \mathcal{M}(p',p')$, which is, by $(i)$, a contradiction. This proves $(ii)$.\\

$(iii)$. By duality, it is enough to consider the case $1 < q < 2 < p < +\infty$. Assume that $\mathcal{M}(q,q) \subset \mathcal{M}(p,p)$. Using the notations introduced in the proof of $(i)$, let $\sigma : \ell_1 \rightarrow \ell_q$ be a quotient map and $J : \ell_q \rightarrow \ell_{\infty}$ be an isometry. Let $\phi \in L^{\infty}(\mathbb{N} \times \mathbb{N})$ be such that
$$u_{\phi} = J I_{\ell_q} \sigma,$$
where $I_{\ell_q} : \ell_q \rightarrow \ell_q$ is the identity map. Then $\phi \in \mathcal{M}(q,q)$. By assumption, $\phi \in \mathcal{M}(p,p)$. By Lemma $\ref{SQ}$, this implies that $I_{\ell_q} \in SQ_p(\ell_q, \ell_q)$. Clearly, this implies that $\ell_q$ is isomorphic to an $SQ_p$-space. But $\ell_q$ does not have type $2$ and any $SQ_p$ has type $2$. This is a contradiction, so $\mathcal{M}(q,q) \nsubseteq \mathcal{M}(p,p)$.
\end{proof}

\begin{theorem}
We have $\mathcal{M}(q,q) \subset \mathcal{M}(p,p)$ if and only if $1 \leq p \leq q \leq 2$ or $2 \leq q \leq p \leq +\infty$.
\end{theorem}

\begin{proof}
By Proposition $\ref{inclusions}$ and duality, we only have to show that when $1 \leq p \leq q \leq 2$, $\mathcal{M}(q,q) \subset \mathcal{M}(p,p)$.\\
We saw in the proof Proposition of $\ref{inclusions}$ $(iii)$ that if $\mathcal{M}(q,q) \subset \mathcal{M}(p,p)$ then $\ell_q$ is isomorphic to an $SQ_p$-space. The converse holds true. Indeed, assume that $\ell_q$ is isomorphic to an $SQ_p$-space. Then by approximation, any $L^q$-space is isomorphic to an $SQ_p$-space. Hence any element of $\Gamma_q(\ell_1,\ell_{\infty})$ factors through an $SQ_p$-space. By the lifting property of $\ell_1$ and the extension property of $\ell_{\infty}$, this implies that any element of $\Gamma_q(\ell_1,\ell_{\infty})$ factors through an $L^p$-space, that is $\Gamma_q(\ell_1, \ell_{\infty}) \subset \Gamma_p(\ell_1,\ell_{\infty})$. By Corollary $\ref{Schurfacto2}$, this implies that $\mathcal{M}(q,q) \subset \mathcal{M}(p,p)$.\\

Assume that $1 \leq p \leq q \leq 2$. By \cite[Theorem 6.4.19]{Kalton}, there exists an isometry from $\ell_q$ into an $L^p$-space, obtained by using $q-$stable processes. Hence, $\ell_q$ is an $SQ_p$-space. This concludes the proof.

\end{proof}

\begin{Problem}
Compare the other spaces of Schur multipliers. For example, if $1 < p \leq 2$, do we have
$$\mathcal{M}(p,1) = \mathcal{M}(p,p)?$$
\end{Problem}

\vskip 1cm
\noindent
{\bf Acknowledgements.} 
The author was supported by the French 
``Investissements d'Avenir" program, 
project ISITE-BFC (contract ANR-15-IDEX-03).

\vspace{1cm}

\end{document}